\documentclass[11pt,letterpaper,reqno]{amsart}
\usepackage[utf8]{inputenc}
\usepackage{amsmath}
\usepackage{amssymb}
\usepackage{datetime}
\usepackage{fancyhdr}
\usepackage{graphicx}
\usepackage{latexsym}
\usepackage{pinlabel}
\usepackage{fullpage}
\usepackage{enumitem}
\usepackage{hyperref}
\usepackage{todonotes}

\usepackage[capitalise]{cleveref}
\newtheorem*{rep@theorem}{\rep@title}
\newcommand{\newreptheorem}[2]{
\newenvironment{rep#1}[1]{
\def\rep@title{#2 \ref{##1}}
\begin{rep@theorem}}
{\end{rep@theorem}}}
\newreptheorem{theorem}{Theorem}
\newreptheorem{lemma}{Lemma}
\newreptheorem{prop}{Proposition}
\newreptheorem{corollary}{Corollary}

\hypersetup{
    unicode=false,          
    pdftoolbar=true,        
    pdfmenubar=true,        
    pdffitwindow=false,     
    pdfstartview={FitH},    
    pdftitle={},    
    pdfauthor={Author},     
    pdfnewwindow=true,      
    colorlinks=true,       
    linkcolor=red,          
    citecolor=blue,        
    filecolor=magenta,      
    urlcolor=cyan,           
}

\newtheorem{theorem}{Theorem}[section]
\newtheorem{lemma}[theorem]{Lemma}
\newtheorem{proposition}[theorem]{Proposition}
\newtheorem{prop*}{Proposition}
\newtheorem{corollary}[theorem]{Corollary}

\theoremstyle{remark}
\newtheorem{remark}[theorem]{Remark}
\newtheorem{claim}[theorem]{Claim}

\def \bdry {\partial}

\def \x {\times}

\newcommand{\CC}{{\mathbb C}}

\newcommand{\ZZ}{{\mathbb Z}}

\newcommand{\CP}{{\mathbb CP}^2}
\newcommand{\bCP}{{\overline{\mathbb CP}}^2}
\DeclareMathOperator{\MCG}{MCG}
\DeclareMathOperator{\Ab}{Ab}
\DeclareMathOperator{\Sym}{Sym}
\DeclareMathOperator{\ind}{ind}

\def \A {\mathcal{A}}
\def \t {\tau}
\def \b {\beta}

\newcommand{\bit}{\begin{itemize}}
\newcommand{\eit}{\end{itemize}}

\newcommand{\ben}{\begin{enumerate}}
\DeclareRobustCommand{\een}{ \end{enumerate} }

\title{Splitting Algebraic Singular Fibrations via Perturbation of Branch Covers
}
\author[S.\ Sakall{\i}]{S\"umeyra Sakall{\i}}
\email{ssakalli@uark.edu}\urladdr{https://sites.google.com/umn.edu/ssakalli/home}
\address{Department of Mathematical Sciences, University of Arkansas, Fayetteville, AR 72701, USA}
\author[Jeremy Van Horn-Morris]{Jeremy Van Horn-Morris}\email{jv002@uark.edu}\urladdr{https://jv002.hosted.uark.edu}
\address{Department of Mathematical Sciences, University of Arkansas, Fayetteville, AR 72701, USA}

\begin{document}
\begin{abstract} 
In a previous paper \cite{SV}, the authors studied the isolated singular fibers that can occur in algebraic fibrations of certain genus two fibrations. There the goal was to determine their monodromy factorizations with the goal of determining a dictionary between a set of curve configurations and certain words in the mapping class group. Each such curve configuration was originally cataloged by Namikawa and Ueno \cite{NamikawaUeno-list} in their list of genus two fibrations. We studied four families of polynomials, we restricted to fibrations whose fibers have boundary, considering the isolated affine singularity referenced in \cite{NamikawaUeno-list}. We resolved the singularities and, using carefully chosen perturbations, deformed them into Lefschetz fibrations and determined their monodromy factorizations. In two of those families, we also obtained strong information about how the central fiber compactifies in a fibration with closed fibers. In this paper we work on the other two cases by recreating the singular fibers using a different family of polynomials. In \cite{SV}, all the algebraic curves in the fibrations were given expressly as hyperelliptic equations. We utilized this symmetry both to construct the deformation and to recover the monodromy factorization. In this paper the curves are no longer expressly hyperelliptic--the quotient curve is no longer just $\CC$, the branch curves are no longer embedded in $\CC^2$, and the fibrations utilized in the quotient are now more complicated. We do, though, recover the desired behavior of the compactification of the central fiber, its deformation to a Lefschetz fibration, and the corresponding monodromy factorization.

\end{abstract}


\maketitle
\section{Introduction}

In \cite{SV} we studied the resolution spaces of four different families, $\phi_1, \phi_2, \phi_3, \phi_4$, of singular algebraic varieties. These spaces admit genus two fibrations each with one singular fiber. We showed that they also admit Lefschetz fibrations whose generic fibers are genus two surfaces (with either one or two boundary components depending on the family). For each member of the four families, we found a flat deformation from the resolution to the Lefschetz fibration and determined the corresponding monodromy factorization. For the $\phi_2$ family, which consists of the resolution spaces of the zero sets of polynomials $y^2 = x(x^4 + t^k)$, $k = 1,\dots, 8$, the fiber is of genus two and has one boundary component. To find the monodromy factorizations we worked with deformations that were compatible with the branched double covers coming from the hyperelliptic involution. Each curve in the fibration is hyperelliptic, and for this family, the fiber over $t=1$ is the curve $y^2 = x(x^4+1)$, which is the branched double cover of $\CC$ branched over the points $x(x^4+1)=0$. The singular variety covers $\CC^2$ and the branch locus was singular. The process of finding the deformation involved certain constrained deformations of the branch locus, either resolving or simplifying the singularities. While there is certainly a deformation of the branch locus to a smooth curve, the corresponding deformation of the singular variety is not always the resolution. As we are interested in studying the resolution, we took particular care when selecting the deformation.

In \cite{NamikawaUeno-list}, Namikawa and Ueno suggest a list of 120 different singular fibers that can arise in a family of genus two curves. They restrict to the projective case. Their initial paper lists an affine polynomial for each singular fiber with an implied claim that the family of projective curves can be resolved so that the central fiber is the curve configuration given in the list. This family has a rather complicated singularity along the points in the projective closure. In the second part of their work, \cite{NamikawaUeno-long}, Namikawa and Ueno construct the first 18 singular fibers via a very different method. These singularities all are associated with periodic symmetries and they first construct a curve with the correct symmetries, then form a quotient corresponding to the specific singular fiber, and then resolve the quotient singularities using Hirzebruch-Jung. After this, they arrive at the correct curve configuration for the central fiber. Matsumoto and Montesinos \cite{MatsumotoMontesinos-book} follow this second path and, using ideas of Nielsen and Thurston to make sense of the quotient orbifolds and Hirzebruch-Jung to resolve, see how the foliations on the product space at least topologically extend to yield the central fiber shown in Namikawa-Ueno. Their analysis holds for all types of pseudo-periodic diffeomorphisms of the closed genus 2 surface. However, their construction is topological rather than algebraic and not particularly amenable to the tools of deformation. 

To get around this complication, we instead work in the affine setting. In the previous work, \cite{SV}, we started with the polynomials suggested by Namikawa and Ueno. We used an algorithm by Nemethi to obtain the resolution. We then found flat deformations of the fibrations and determined their monodromy factorizations, using tools from contact topology to invoke Laufer's theorem on simultaneous resolutions. These resolution graphs gave us the required information to determine the monodromy factorizations for the four families. 

In this paper, we consider the family $\psi_1^k: t^k = x(x^3-y^2)$, $k = 1,\dots, 8$, and refer to it collectively as the $\psi_1$ family. We calculate the resolutions using Nemethi's algorithm where the covering now occurs over the central fiber in the fibration. The curves that correspond to fixed values of $t$ are no longer hyperelliptic in the classical sense but they do admit an involution whose quotient is a rational curve (in this case, biholomorphic to $\CC^\x$). Hence each singular variety can be built as the double branched cover over an annulus fibration branched over some braided multisection. We determine the correct deformations of the section to induce a flat deformation of the resolution of the singular surface so that the induced fibration is Lefschetz, splitting the more complicated fiber over $t=0$ into Lefschetz singularities. We then find their corresponding monodromy factorizations. Because we again care about studying the fibration on the resolution, we have to take particular care in choosing which deformation to make. Care was made to determine that the resulting deformation of the branch locus lifts to a flat deformation of the resolution. First we determine the topological type of the deformed surface, then we appeal to Laufer (see Theorem~\ref{thm:laufer}) and use some tools from contact topology to show that Laufer's theorem can be applied (see Proposition~\ref{prop:SV}). 

For the families $\phi_1$ and $\phi_3$, a surprising benefit of Nemethi's algorithm was that it also gave interesting hints that the projective curve configuration of Namikawa Ueno was indeed the resolution of the projectivization. Knowing how the fibration on the affine picture embeds in the projective picture tells you how the sections at infinity interact with the fibration and how they intersect the central fiber. The output of the resolutions in the affine case gave a result that was tantalizingly close to the exact central fiber in the projective fibration used by Namikawa and Ueno, in that its compactification, capping off the one or two noncompact components with disks, is exactly the configuration of the projective case. However Nemethi's algorithm did not apply in the same way to the families $\phi_2$ and $\phi_4$. (Nor, for our original purposes, did it need to.) Nemethi's algorithm uses a cyclic action and the resulting plumbing diagram includes the information about the fixed point locus of the action. For the families $\phi_1$ and $\phi_3$, the cyclic action is a rotation around the singular fiber and Nemethi's algorithm then returns the central fiber of the resolution with a complete description of the multiplicities of the covering by the smooth fiber on both the compact and non-compact components. The latter of which will hold the key to understanding how sections of the fibration intersect the singular fiber. For the families $\phi_2$ and $\phi_4$, we instead used the involution on the singular space induced by the hyperelliptic involution of the fibers, and found their monodromy factorizations. Because the fixed point locus of this action is not the central fiber, Nemethi's algorithm tells us less about the central fiber of the resolution. We did not access to the non-compact components of the central fiber and did not see how the affine fibration embeds in the projective one.

With that in mind, in this paper we find a new family of polynomials $\psi_1^k$ that recover the central singular fibers at $\phi_2^k$ for each $k=1, \cdots, 8$, (fibers of type VII, VI, VII, I*, VII*, VI, VII*, and the generic fiber $\Sigma_2$, respectively). The multiplicities of each irreducible component of $\psi_1^k$ match up with those of Namikawa-Ueno fibers, and our polynomials are different than what Namikawa-Ueno have. We also note the relationship between $\phi_2^k$ and $\psi_1^k$. The resolution graph for $\phi_2^k$ is a subset of the resolution graph for the corresponding $\psi_1^k$ for each $k=1, \cdots, 8$.

In outline, after giving background in Section \ref{background}, we find the resolution of the single compound singularity of each $\psi_1^k$ over the origin in Section~\ref{sec:resolutions}. Then we use these to prove our main Theorem \ref{thm:main2} below, the proof of which we complete in Section \ref{psik}.

\begin{theorem} \label{thm:main2}
The genus two fibration on the resolution $X$ of a singular algebraic variety $V(f)$ where $f$ is 
\[t^k = x(x^3-y^2), \; k = 1,\dots, 8,\]

splits into a Lefschetz fibration described by one of the following positive words in the mapping class group of the genus two surface with 2 boundary components:

\begin{itemize}
\item $\psi_1 = \tau_4 \tau_3 \tau_2 \tau_1 \tau_{1'}$, 
\item $\psi_1^2 = (\tau_4 \tau_3 \tau_2 \tau_1 \tau_{1'})^2$,
\item $\psi_{\tilde{1}} = \tau_2 \tau_3 \tau_4 \tau_5 \tau_{5'} \tau_{\partial_1}$,  
\item $\psi_1^4=\psi_{\tilde{1}} \psi_{{1}}$, 
\item $\psi_1^5= \psi_{\tilde{1}} \psi_{{1}}^2$, 
\item $\psi_1^6 = \psi_{\tilde{1}}^2$, 
\item $\psi_1^7=\psi_{\tilde{1}}^2\psi_{{1}}$, 
\item $\psi_1^8=\tau_{\partial_1}^3\tau_{\partial_2}$
\end{itemize} 
where the labeling agrees with the labeling of the Dehn twist curves on the surface $\Sigma_{2,2}$ shown in Figure~\ref{fig:branched cover} and $\tau_{\bdry_{i}}$ stand for the boundary multitwists on the surface $\Sigma_{2,2}$.
\end{theorem}

Here, we consider factorizations in the mapping class group $$\mathrm{MCG}(\Sigma_{2,2}) := \mathrm{Diff}^+(\Sigma_{2,2}, \bdry  \Sigma_{2,2})/\text{isotopy rel boundary}.$$ The Dehn twists $\tau_1, \tau_{1'}, \tau_2, \cdots, \tau_5, \tau_{5'}$ are the standard generators of the hyperelliptic subgroup of the mapping class group of the genus two surface with 2 boundary components as shown in Figure~\ref{fig:branched cover}. The monodromy $\psi_1$ is a root of the boundary twist of order $8$.

We also give a new lift of the hyperelliptic involution to the mapping class group of $\Sigma_{2,2}$ (see Section \ref{sec:hyperelliptic} for the proof):

\begin{proposition}\label{newlift} The following relation holds in the mapping class group of the surface $\Sigma_{2,2}$
\[(\tau_2 \tau_3 \tau_4 \tau_5 \tau_{5'} \tau_4 \tau_3 \tau_2 \tau_1\tau_{1'})^2 = \tau_{\partial_1} \tau_{\partial_2}\]
and moreover, the mapping class element represented by $\tilde{I} = \tau_2 \tau_3 \tau_4 \tau_5 \tau_{5'} \tau_4 \tau_3 \tau_2 \tau_1\tau_{1'}$ is isotopic to the involution on $\Sigma_{2,2}$ with four fixed points whose quotient is the annulus. 
\end{proposition}

\section{background}
{\bf Resolutions and deformations} 
\label{background}

Our goal in this paper is to understand a singular fibration on the resolution of certain hypersurface singularities $V(f) \subset \CC^3$, for the polynomials $f(x,y,t) = t^k-x(x^3-y^2)$. Most of the varieties $V(f)$ are singular, having an isolated singularity at the origin. If we intersect the variety $V(f)$ with the hyperplanes $t=\mathit{const}$ we get a fibration of $V(f)$ by the curves $x(x^3-y^2)=t^k$ parametrized by $t$, each contained in the corresponding hyperplane. Each of these curves is noncompact, of genus 2, and with ends (which we will typically treat as boundary components). The fiber over $t=0$ is singular, but fibers are otherwise smooth. In Section \ref{sec:resolutions}, we calculate the resolutions, $X = X_f$, of each of these hypersurfaces using an algorithm of Nemethi \cite{Nemethi-lectures}. The resolution inherits the fibration above and the \emph{central fiber}, the fiber over $t=0$, encodes the resolution \cite{mumford}. We record the topological type of the central fiber as a plumbing diagram with arrowheads to indicate the noncompact components. (The results are summarized in Figures \ref{VIIalt}, \ref{VIalt}, \ref{VII*alt}, \ref{phi2differentpoly}, \ref{table2}.) 

We additionally want a description of the deformation of this fibration into a \emph{Lefschetz fibration}. Rather than having a single complicated singular fiber at $t=0$, a Lefschetz fibration has many singular fibers which are all nodal curves, as nice a singularity as one could hope. One benefit of working with a Lefschetz fibration is that, up to symplectic deformation, they are uniquely determined by their corresponding \emph{positive monodromy factorization}, but a second is that they are a ubiquitous \cite{Donaldson} tool for studying symplectic 4-manifolds and there is a large literature of computational techniques and applications. Genericity says that our original singular fibration can be perturbed to a Lefschetz fibration, but care is required to keep track of the relative positions of the nodes and recover the corresponding monodromy factorization. We choose to do this directly via a deformation.

The particular (sometimes frustrating) difficulty when using deformation in the problem above is that significant care must be utilized to ensure that the deformation preserves the symplectic (or even topological) type of $X$. Indeed, for all of our polynomials, there is a deformation from $V(f)$ to some smooth hypersurface of $\CC^3$ given by some generic perturbation of $f$. This induces a deformation of the resolution $X$, and its corresponding fibration, to some smoothing of $V(f)$. And this occurs even though typically $X$ and the smoothing are not even homeomorphic. 

To deform the fibration on $X$, we use a symmetry of the fibration inherent in its defining polynomial to describe both $V(f)$ and $X$ as the double branched cover of some other hypersurface in $\CC^3$, branched over a curve $\Delta$ which is transverse to the same fibration by $t=\mathit{const}$. We deform $\Delta$ within the hypersurface and resolve each of the singularities of the cover. This is the family that we can control and that we can guarantee gives a flat deformation $X_s$ of $X$ into the Lefschetz fibration by choosing the correct deformation $\Delta_s$ of $\Delta$.

To prove that the given deformation is flat, we utilize the following theorem of Laufer:

    \begin{theorem}[Laufer \cite{Laufer-weak}, Theorem 5.7]  \label{thm:laufer} Let $\lambda: \mathcal{V} \to T$ be the germ of a flat deformation of the normal Gorenstein two-dimensional singularity $(V,p)$, with $T$ a reduced analytic space. Then $\lambda$ has a very weak simultaneous resolution, possibly after finite base change, if and only if $K_s \cdot K_s$ for $s \in T$ is constant. 
    \end{theorem}

Throughout, we will use Laufer's theorem to construct a very weak simultaneous resolution of a deformation of the singularities in question. The deformation of varieties $f_s$ that we constructed above yields a flat deformation $V(f_s)$ of the two-dimensional singularity $V = V(f)$ and we calculate $K_s \cdot K_s$ by its value on the corresponding resolution $X_s$ of $V(f_s)$. 

The variety $V$ is a hypersurface and hence also Gorenstein. Each $X_s$ comes with a fibration by algebraic curves induced by the projection onto the $t$ coordinate. Truncating to a disk of large radius gives the boundary the structure of an open book $B$ on $Y = \bdry X_s$, which is independent of $s$, and each $X_s$ is pseudo convex with boundary a plane field isotopic to the contact structure $\xi$ carried by $B$. For each $X_s$, then we have 
\[K_s \cdot K_s = K(X_s)^2=c_1^2(X_s) = 4 d_3(\xi)+2\chi(X_s) + 3 \sigma(X_s) \label{eqn:gompf}\]
where $\chi(X_s)$ is the Euler characteristic and $\sigma(X_s)$ is the signature, and $d_3$ is Gompf's 3-dimensional invariant of the plane field $\xi$ \cite{Gompf}. As $d_3$ is determined by $\xi$, it is independent of the choice of deformation. Thus a deformation determines a splitting precisely when the Euler characteristic and signature of $X_{f_s}$ is independent of $s$.

This implies the following proposition:

\begin{proposition}[\cite{SV}] \label{prop:SV} Let $(V,p)$ be an isolated two-dimensional singularity whose link is strictly pseudoconvex and with resolution $X$ which is symplectically deformation equivalent to a Stein domain. Let $V_s$ be a flat deformation of $(V,p)$ with isolated singularities and strictly pseudoconvex boundary and having minimal resolutions $X_s$, also symplectically deformation equivalent to a Stein domain. Then the family $X_s$ forms a flat deformation of $X$ (possibly after finite base change) if and only if the $X_s$ share the same $b_1$, $b_2^+$ and $b_2^-$ as $X$. 
\label{prop:Laufer}
\end{proposition}

In all of our examples, the intersection form of $X$ will be negative semi-definite and $b_1(X)=0$. A deformation of the resolution of the singularity will be flat if the intersection form of $X_S$ is also negative definite, $b_1(X_s) = 0$ and $b_2(X_s) = b_2(X)$. 

As discussed, only certain restricted choices of deformations $\Delta_s$ (or $X_s$) guarantee that this condition holds and we have to ensure that our deformations conform to these constraints. The values of $b_1$, $b_2$ and the negative definite property will be checked in each case.

{\bf{Mapping class groups and branched covers}} 

Let $S$ be a surface with boundary and let $S^*$ denote $S$ with $n$ marked points, $p_1, \dots, p_n$. Let $\MCG(S)$ denote the mapping class group of $S$, group of isotopy classes of orientation preserving diffeomorphisms which are the identity on $\partial S$. Let $\MCG(S^*)$ denote the mapping class group of $S^*$, the group of isotopy classes\footnote{Isotopies should also fix the boundary and preserve the set of marked points.} of diffeomorphisms of $S$ which fix $\partial S$ pointwise and which preserve the set of marked points $\{p_1, \dots, p_n\}$. The group $\MCG(S)$ is generated by Dehn twists \cite{Lickorish} and the group $\MCG(S^*)$ is generated by Dehn twists and braid half-twists \cite{Lickorish, FarbMargalit}. We will sometimes refer to $\MCG(S^*)$ as the \emph{generalized braid group} of S.

If $S_1 \mapsto S_2$ is an $m$-fold branched cover with ramified points $p_1, \dots, p_n$, then a diffeomorphism $\phi_2\in \MCG(S_2^*)$ lifts to a diffeomorphism of $S_1$ if and only if it commutes with the deck action $\pi_1(S_2 -\{p_1, \dots, p_n\} \mapsto \Sym(m)$ \cite{BirmanHilden}. This is independent of $\phi_2$ up to isotopy and closed under composition. Birman and Hilden called the corresponding subgroup of the mapping class group the \emph{liftable} subgroup and additionally characterized which powers of Dehn twists and braid half-twits lift to diffeomorphisms of $S_1$ (as well as what they lift to). As the fundamental example, a braid half-twist between $p_i$ and $p_j$ along an arc $a$ lifts if and only if the preimage of $a$ consists of a single connected circle which double covers $a$ and some number of arcs which cover $a$ with degree 1. In this case, braid half-twist lifts to a Dehn twist along the circle component. In this paper, we will deal only with double branched covers, so the lifting condition is guaranteed for all mapping class elements. 

\begin{remark} Our particular interest is in \emph{positive factorizations} in $\MCG(S)$, ways of factoring mapping class elements into products of right-handed Dehn twists $\phi = \t_1 \cdots \t_k$. In this paper, we use \emph{right-action} notation for the action of the mapping class group on the surface, so in the product $\t_1 \cdots \t_k$, $\t_1$ acts first followed by $\t_2$ and so on. Conjugating a Dehn twist by any diffeomorphism again yields a Dehn twist: $\phi^{-1} \t \phi = \phi(\t)$. For the latter notation, if $\t$ is the Dehn twist about the curve $\gamma$ then $\phi(\t)$ is the Dehn twist about its image $\phi(\gamma)$. \end{remark}

A \emph{Hurwitz move} on a positive factorization is a rewriting of the factorization using one of the two following conjugations:

\begin{itemize}
    \item $\cdots \t_i \t_{i+1} \cdots \rightarrow \cdots \t_{i+1} \cdot ( \t_{i+1}^{-1} \t_{i} \t_{i+1}) \cdots = \cdots \t_{i+1} \cdot \t_{i+1}(\t_i) \cdots,$ or
    \item $\cdots \t_i \t_{i+1} \cdots \rightarrow \cdots (\t_{i} \t_{i+1} \t_i^{-1}) \cdot \t_i \cdots = \cdots \t_i^{-1}(\t_{i+1+}) \cdot \t_i \cdots$
\end{itemize}

We will sometimes refer to this as \emph{sliding} one twist past another. We can similarly conjugate braid half-twists with each other, as well as half-twists with Dehn twists and vice versa. (See \cite{BirmanHilden} or Chapter 9 of \cite{FarbMargalit}, for example.) They all have the same relationship: $\t_1^{-1} \t_2 \t_1 = \t_1(\t_2)$.

Our interest in positive factorizations lies in their connection to \emph{Lefschetz fibrations} \cite{Donaldson, Gompf, LoiP}. A Lefschetz fibration is a surjection $\pi : X \rightarrow B$ where $X$ has dimension 4 and $B$ has dimension 2 and the only allowed singularities have the local model of $\pi:(z_1, z_2) \rightarrow z_1^2 + z_2^2$ (in orientation-preserving complex coordinates). We often assume that all singularities lie in distinct fibers of $\pi$ (which can be realized by a small perturbation). Throughout, $X$ will be either non-compact or with boundary and we assume that the critical points of $\pi$ lie in the interior of $X$. The smooth fibers of this fibration are oriented surfaces and each singular fiber is a nodal singularity obtained by collapsing a simple closed curve (the \emph{vanishing cycle}) in the nearby smooth fibers. Away from the singular points, the Lefschetz fibration admits a horizontal connection (which is orthogonal to the fibers in the neighborhood of the singularity) and the vanishing cycle in a nearby fiber consists of all points which limit to the singularity under the horizontal flow. The \emph{monodromy} around a singular fiber is given by a positive (right-handed) Dehn twist along the vanishing cycle and we identify the vanishing cycles of different fibers by choosing a reference fiber (assumed to be smooth) $\pi^{-1}(x_0) = \Sigma_0 = \Sigma_{g,k}$, a surface of genus $g$ and having $k$ boundary components (or ends if the surface is non compact), lying over the reference point $x_0 \in B$. Then for each singular fiber, choose a path from the reference fiber to the singular fiber and extend the vanishing cycle along the path to get a curve in $\pi^{-1}(x_0)$. (This depends on the choice of horizontal connection for the bundle, but only up to smooth isotopy.) A choice of pairwise disjoint, smooth paths gives a cyclic ordering of the vanishing cycles by how they emanate from the reference point. Choosing a disk $D$ in $D^2$ containing all singular values of $\pi$ and with $x_0 \in \partial D$, the preimage of $\partial D$ is a surface bundle over $S^1$. Using the reference fiber, we can identify the monodromy $\phi$ of this surface bundle as an element of $\MCG(\Sigma_0)$. The product of the positive Dehn twists for all vanishing cycles in the cyclic order induced by the choice of paths is a positive factorization of $\phi$ in $MCG(\Sigma_0)$, $\phi = \t_1 \cdots \t_k$. Indeed, any Lefschetz fibration over $D^2$ is determined up to diffeomorphism by its corresponding positive \emph{monodromy factorization}. (See \cite{GS}.)

A \emph{braided surface} in a Lefschetz fibration $\pi: X \rightarrow \CC$ is a properly embedded surface $S \in X$ which is disjoint from the singular points of $\pi$ and which intersects each fiber transversely save for a finite number of tangencies and all transverse intersections are positive. At places of tangency we require that $S$ has a standard form with respect to the fibration. Locally choosing complex coordinates $(x,y)$ on $X$ so that $\pi$ is given by the projection to the first factor, we allow tangencies of the form $x^p = y^q$. If every point of tangency of $S$ occurs at a different fiber and with a local model of $x = y^2$ we say that $S$ is \emph{simply braided}. (See \cite{AurouxKatzarkov}.) We call the image under $\pi$ of all the places of tangency the \emph{braid values} of the surface $S$. The \emph{index} of $S$ is the intersection number of $S$ with a regular fiber. If we take a small enough disk $D$ in $D^2$ containing the braid value of a simple tangency, it is disjoint from the set of singular values of $\pi$ and so the fibration is trivial. If we look at how $S$ intersects each fiber, over the boundary of $D$ it intersects each fiber in $\ind{S}$ points and the movie of those intersections as we move around $\partial D$ is a right-handed braid halftwist $\beta$ between two of those points along some arc $a$. The collection and relative placement of these halftwists determines $S$ (up to isotopy). More specifically, to any braided surface we can associate a quasipositive braid half-twist factorization of its boundary braid. We do this by choosing a reference fiber (typically $x=1$) and for each braid value a choice of pairwise disjoint simple paths in $\CC$ from the reference fiber to the braid value. We additionally choose one path each (pairwise disjoint from each other and also from the braid paths) from the reference fiber to the singular values of $\pi$. We can achieve a Hurwitz move of the factorization by replacing the choice of path $a_i$ in the factorization with the arc that starts to the left of $a_{i+1}$ (or to the right of $a_{i-1}$ in the second case) and which stays in a neighborhood of their union. (See figure Hurwitz.) Because of this equivalence, we typically are interested in (quasi-)positive factorizations only up to Hurwitz equivalence. 

Suppose we are given a simply braided surface $S$ in $X$ and let $\Sigma_0$ be the reference fiber of $\pi$ and $\{q_1, \dots, q_d\}$ the intersection of $S$ with $\Sigma_0$. Let $\Sigma'_0$ be a double cover of $\Sigma_0$ branched over $\{q_1, \dots, q_d\}$. If every vanishing cycle of $\pi$ lifts to $\Sigma'_0$, then there is a double cover $\rho:X' \rightarrow X$, branched over $S$, in which $\Sigma'_0$ is the preimage of $\Sigma_0$ and $\rho|_{\Sigma'_0}$ is our chosen branched double cover. In this case, the composition $\pi' = \pi\circ \rho$ is a Lefschetz fibration on $X'$. The set of singular values of $\pi'$ is the union of the singular values of $X$ and the braid values of $S$. A choice of a set of arcs of $S$ and $\pi$ gives a positive factorization of $S \subset X$ and it also gives a set of arcs for the singular values of $\pi'$. From the picture, we can read off the positive factorization associated to $\pi'$ by these arcs.
\begin{itemize}
\item Each vanishing cycle of $\pi$ lifts to two vanishing cycles in $\pi'$, the preimages.
\item Each braid singularity lifts to a Lefschetz singularity. Its corresponding vanishing cycle is the simple closed curve that is the preimage of the braid half-twist arc associated to that braid value of $S$.
\end{itemize}

(see, for instance, \cite{AurouxKatzarkov, IT}). 

Additionally, we see that Hurwitz moves on $X$ yield (sometimes compound) Hurwitz moves on $X'$ which follow the same lifting pattern. The principal application of this here to determine the positive factorization of a Lefschetz fibration when we can write it as a suitably nice double branched cover, as well as to demonstrate certain simplifications by Hurwitz moves.

\section{Resolutions and Namikawa-Ueno fibers of types VII, VI and VII* from different polynomials} 
\label{sec:resolutions}

In \cite{SV} we have constructed 13 types of the Namikawa-Ueno fibers in a different way than in \cite{NamikawaUeno-long}, but from the same polynomials as given in \cite{NamikawaUeno-long} and \cite{NamikawaUeno-list}. These are resolutions of suspension type singularities and we determined the resolutions by applying Nemethi's algorithm in \cite{Nemethi-algorithm}. However, we would like to remark that Namikawa-Ueno's defining polynomials for the type VII, VI and VII* fibers are not of suspension type and Nemethi's algorithm does not apply to resolve them. We are able to construct type VII, VI and VII* fibers from different polynomials as follows.  

Let us begin with type VII. It is given by the polynomial $y^2=x(x^4+t)$ in \cite{NamikawaUeno-list}. The resolution of $t=x(x^3-y^2)$ also yields a type VII singular fiber: Following Nemethi's algorithm, first we resolve the plane curve singularity $x(x^3-y^2)=0$. That gives us the first configuration $\mathcal{S}$ in Figure \ref{VIIalt}. We note that this is the same configuration as in \cite{NamikawaUeno-long}, the first figure on p.342. Then, we apply Nemethi's algorithm with N=1; i.e., we take the degree one cover of $\mathcal{S}$. Next, we successively blow-down the $-1$ spheres as shown in Figure \ref{VIIalt}. We obtain type VII fiber at the end, after capping-off the arrows. (Compare this to \cite{NamikawaUeno-long}, pp.342-343.) We also note the typo in \cite{NamikawaUeno-list}, in the configuration of type VII fiber. The cuspidal curve is a rational curve as they construct and write in \cite{NamikawaUeno-long}, it is not a genus one curve.

\begin{figure}[htb]
{\includegraphics[width=15cm]{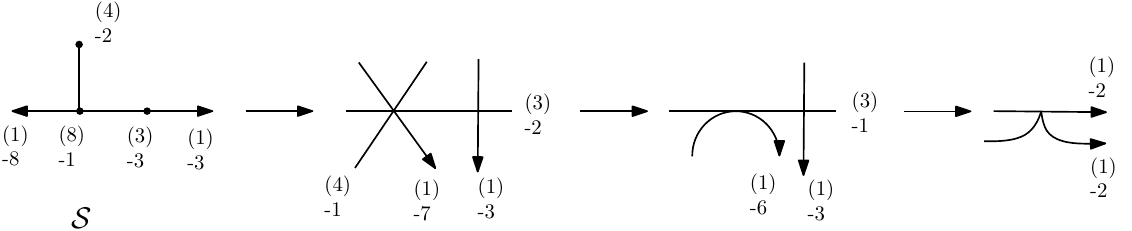}}
\caption{Construction of fiber VII from a different polynomial: $\psi_1$.}
\label{VIIalt}
\end{figure}

Type VI fiber is given by the polynomial $y^2=x(x^4+t^2)$ in \cite{NamikawaUeno-list}. We reconstruct it from $t^2=x(x^3-y^2)$. We apply Nemethi's algorithm to the configuration $\mathcal{S}$ above, with $N=2$. As a result, we obtain a minimal graph as we show in Figure \ref{VIalt} which is the fiber of type VI.

\begin{figure}[htb]
{\includegraphics[width=9cm]{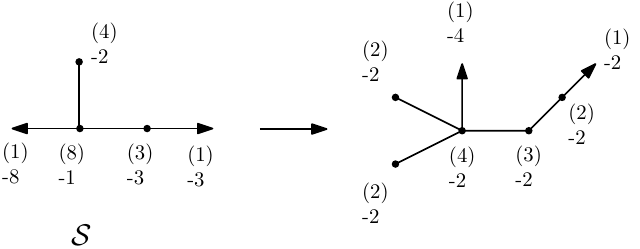}}
\caption{Construction of fiber VI from a different polynomial: $\psi_1^2$.}
\label{VIalt}
\end{figure}

Type VII* is given by the polynomial $y^2=x(x^4+t^5)$ in \cite{NamikawaUeno-list}. We reconstruct it from $t^5=x(x^3-y^2)$. We apply Nemethi's algorithm again to the configuration $\mathcal{S}$, with $N=5$. As a result, we obtain the first graph in Figure \ref{VII*alt}. Next, we blow-down the $-1$ spheres, this gives us type VII* fiber as depicted in Figure \ref{VII*alt}.

\begin{figure}[htb]
{\includegraphics[width=15cm]{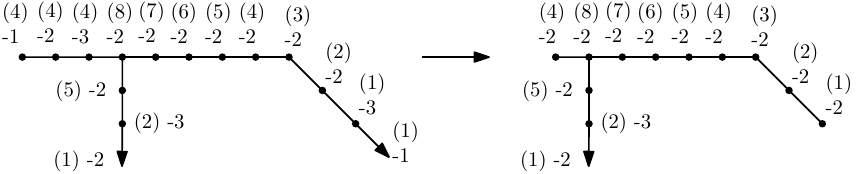}}
\caption{Construction of fiber VII* from a different polynomial: $\psi_1^5$.}
\label{VII*alt}
\end{figure}

We also resolve the remaining polynomials $t^k=x(x^3-y^2)$, for $k=3,4,6,7,8$ in the same way as we discussed above. We summarize the results in Figure \ref{phi2differentpoly}. The second column shows the graphs when we apply Nemethi's algorithm to the configuration $\mathcal{S}$ for the given values of $k$, and the third column shows the resulting graphs when we blow-down the $-1$ curves.

\begin{figure}[htb]
{\includegraphics[width=15cm]{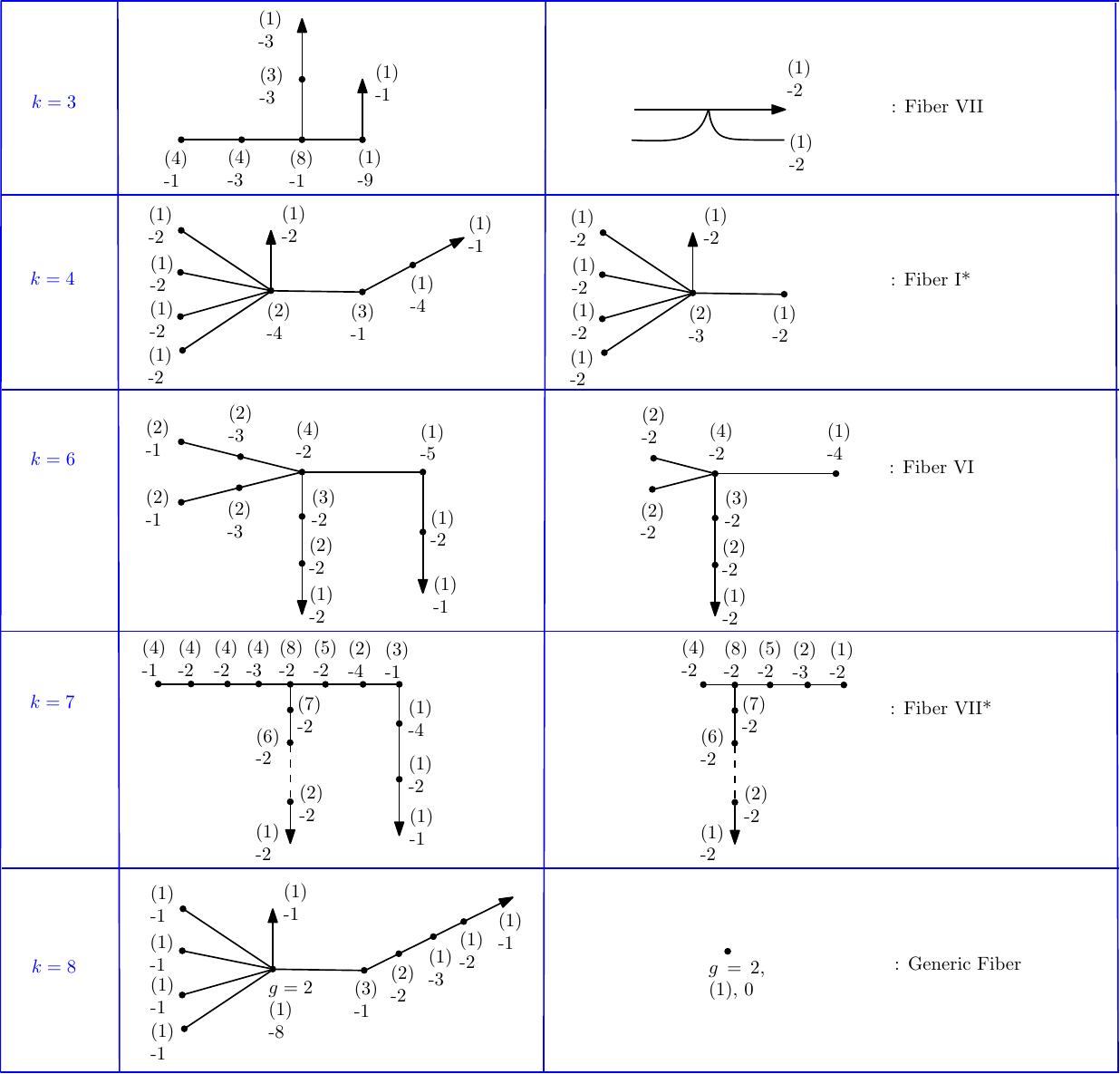}}
\caption{Constructions of VII, I*, VI, VII* from different polynomials.}
\label{phi2differentpoly}
\end{figure}

Next, we would like to remark that in the resolution graphs of $t^k=x(x^3-y^2)$ for $k=3, \cdots, 8$, there are $-1$ arrows (cf. the graphs in the first column of Figure \ref{phi2differentpoly} and the first graph in Figure \ref{VII*alt}. (Also, note that Figure \ref{VII*alt} is the $k=5$ case)). When we do not cap these off and thus do not blow them down, we obtain the graphs as shown in Figure \ref{table2}. In computing the monodromy factorizations for each $k$ in the following section, we will use these graphs in Figure \ref{table2} (for the $k=3, \cdots, 8$ cases) and the ones in Figures \ref{VIIalt} and \ref{VIalt} (for $k=1$ and $k=2$, respectively).

\begin{figure}[htb]
    \centering
{\includegraphics[width=15cm]{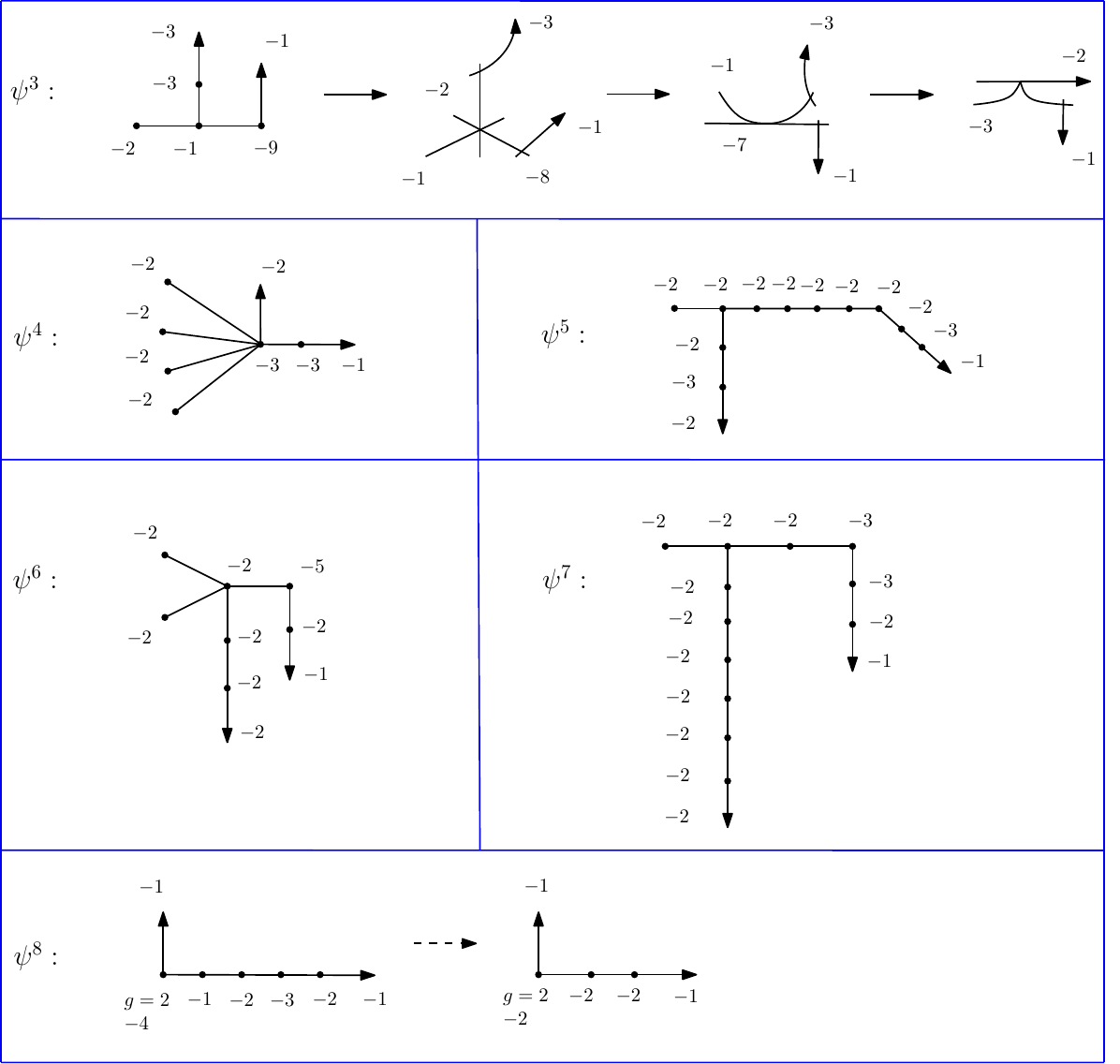}}
\caption{Resolution graphs of $t^k=x(x^3-y^2)$ for $k=3, \cdots, 8$ when we do not cap off the arrows.}
\label{table2}
\end{figure} 

\subsection*{Resolution of \texorpdfstring{{\boldmath $t^3 = x(x^2-y^2)$}}{sec:g=1}}
\label{sec:g=1}

Now we resolve $t^3 = x(x^2-y^2)$ and we will use it in Section \ref{sec:psi0}. First we resolve the plane curve singularity $x(x^2-y^2)=0$ by blow-ups. This gives one vertex $w$ of multiplicity 3, self intersection $-1$, and 3 arrows emanate from $w$. Next, to resolve $t^3 = x(x^2-y^2)$ we apply Nemethi's algorithm with $N=3$. We find that above $w$ there is 1 vertex $w'$ of multiplicity 1 and genus 1. Above each arrow emanating from $w$, there is one string of type $G(3,1,3)$ (see \cite{Nemethi-lectures}, Appendix 1 for the string notation). After computing we see that these strings are arrows, they start at $w'$, and have no extra vertices. Then we compute that the central vertex $w'$ has self intersection $-3$. This completes the resolution.

\section{Deformations and monodromy factorizations}
\label{psik}

\subsection{Introduction and base case}
\label{sec:basecase}

Our goal is to understand the singular fibration on the vanishing locus $V(f)$ of the polynomial $\psi_1: t = x(x^3-y^2)$ in $\CC^3$. $V(f)$ is smooth and so biholomorphic its resolution $X$ which is in this case $\CC^2$. As a subvariety of $\CC^3$, $X$ is the graph of the function $f(x,y) = x(x^3-y^2)$ and the fibration is given by intersecting with the planes $t=\mathit{const}$, giving us a map $\pi_t: X \rightarrow \CC_t$ induced by projection on $\CC^3$. For each value of $t$, then, the fiber in $X$ is the curve in $\CC^2$ given by $f(x,y) = t$, and indeed the entire fibration is the Milnor fibration on $\CC^2$ given by the polynomial $f(x,y)$. The singular fiber over $t=0$ is complicated and its compact components are shown in the resolution in Figure~\ref{VIIalt}. Every algebraic fibration is close to a Lefschetz fibration and our goal is to realize such a deformation of the fibration and determine the resulting monodromy factorization. 

The fibers $\mathcal{C}_t$ of $\pi_t$ are genus 2 curves with two boundary component. To see this, we look at the link $\mathcal{L}$ of a fiber $x(x^3-y^2) = \epsilon$ in $S^3.$ This is a two component link, consisting of an unknot and the trefoil, where the trefoil is braided about the unknot as the (2,3) torus knot about the braid axis. (See Figure~\ref{fig:braid1}.) This is an oriented, fibered, strongly quasipositive $\CC$-link (in the terminology of Rudolph), with self-linking number 4 (see Rudolph), and so every complex curve with boundary $\mathcal{L}$ has genus 2 \cite{Rudolph}.

\begin{figure}
    \centering
    \includegraphics[width=5.5in]{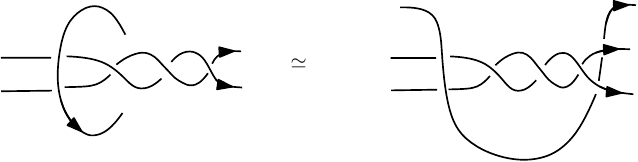}
    \caption{A braid description for the link of $x(x^3-y^2) = \epsilon$ as a subset of $S^3$.}
    \label{fig:braid1}
\end{figure} 

The entire Milnor fibration on $\CC^2$ is invariant under the involution $(x,y) \mapsto (x,-y)$. The quotient space of $\CC^2$ under the involution is also $\CC^2 =<x,z>$ where $z=y^2$ and the complex curves $\mathcal{C}_t$ quotient to the curves $\mathcal{A}_t$ defined by $t = x(x^3 -z)$. The latter curves are all annuli and for $t\neq0$ can be parametrized as $z = \frac t x -x^3$ for $x \in \CC^\times$. The involution fixes the hyperplane $y=0$ which maps to the plane $z=0$ in the quotient. Each of the curves $\mathcal{C}_t$ for $t\neq 0$ is transverse to the fixed point set and so the involution gives a branched double cover $\mathcal{C}_t \mapsto \mathcal{A}_t.$ The branch points in $\mathcal{A}_t$ correspond to the intersections with $z=0$, which correspond to the four solutions to $x^4 = t$. This cover is shown in Figure~\ref{fig:branched cover} for $t=1$.

\begin{figure}
    \centering
    \includegraphics[width = 6in]{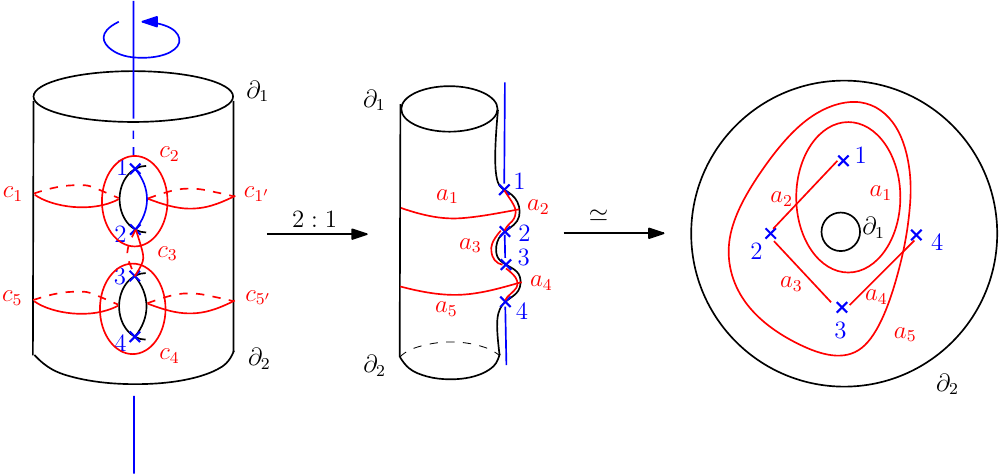}
    \caption{The branched cover of $\mathcal{C}_1$ over $\A_1$ and identification with $D^2 -\{0\}\subset \CC^\times$. We also identify the arc and circle collection in $\A_1$ which we will use and their corresponding lifts to $\mathcal{C}_1$. The Dehn twists that we will use in all our factorizations are $\tau_i$ for $i = 1, 1', 2, 3, 4, 5, 5', \partial_1, \partial_2$ about the curves $c_i$ (or $\partial_j$) shown in the figure.}
    \label{fig:branched cover}
\end{figure}

We can read off the monodromy of the fibration around a small disk containing $t=0$ by looking at the monodromy of the branch locus in the quotient. We do this by following the solutions $x^4 = t$ in $\CC^\times$ as $t$ traverses the circle: $t = \epsilon e^{2\pi i \theta}$ for $\theta \in [0,2\pi]$ and as there are no singularities away from $t=0$, we can choose $\epsilon =1$. At $\theta = 0$, the four solutions are $x = \pm1,\pm i $. As $\theta$ increases to $2\pi$, the four branch points make a counterclockwise rotation through $\pi/2$ radians. Lifting this monodromy to the double branched cover gives the monodromy of $f(x,y)$. With a bit of work (see \cite{Keiko}), we could construct a Dehn twist factorization of this monodromy starting with a factorization of the braided branch locus into liftable twists, but instead we achieve a stronger result by finding the positive factorization associated to a Lefschetz fibration close to the Milnor fibration.

The Milnor fibration is not a Leftschetz fibration but we would like to deform it into one. We can achieve this by the perturbation $t = f_s(x,y) = x(x^3 -s - y^2)$. This is again a fibration on a nearby copy of $\CC^2$ in $\CC^3$. As before, we can think of the fibration being given either by intersecting with the hyperplanes $t = const$ or via the orthogonal projection to the $t$ coordinate. For $s$ small, non-zero, there are four Lefschetz singularities: three at the values of $-3(s/4)^{4/3}$ and one over 0. Both the surface and the fibration are equivariant with respect to the involution $y\mapsto -y$. The quotient is the fibration $t = x(x^3 -s - z)$ on $\CC^2 \subset \CC^3$ and the branch locus corresponds to $y=0$ upstairs and $z=0$ downstairs. We reuse the above notation and again call the fibers of $f_s$ $\mathcal{C}_t$ and the fibers in the quotient $\A_t$. The branch curve $z=0$ we denote by $\Delta_s$. Indeed, using our understanding of braided surfaces in $\CC^2$, we can read off the monodromy by analyzing the behaviors of the branch points $\Delta_s \cap \mathcal{A}_t$. The three Lefschetz singularities over the values of $-3(s/4)^{4/3}$ correspond to simple branching of branch surface $\Delta_s$, thought of as a braided surface in the annulus fibration $f_s$ on $\CC^2$. This fibration has a nodal singularity at $t=0$ that lifts to two nodal singularities in the double cover. To understand the braiding, we parameterize the fibers for $t\neq0$ by 
$x \in \CC^\times$ using the map $z = x^3 - s -t/x$. The branch points $\Delta_s \cap \mathcal{A}_t$ correspond to the four solutions to $z=0$, specifically $0 = x^4 - sx - t$.

To determine the final factorization, we want to identify the vanishing cycles for each of the Lefschetz singularities as well as the vanishing cycle in $\Delta_s$ that produces the nodal singularity. We measure everything relative to the reference fiber over $t=1$. Over the fiber $\A_1$, the branch points are very near the fourth roots of unity. For $s$ a small positive real number, the three Lefschetz critical values in $\CC$ have angular coordinate $\pi/3, \pi, 5\pi/3$. Using the paths shown in Figure~\ref{fig:psi1} from $t=1$ to the Lefschetz critical values, we identify the three Lefschetz vanishing cycles as the double cover of three arcs indicating the braiding of the surface $\Delta_s$.

The vanishing cycle in $f_s$ contributing to the nodal singularity over $t=0$ can be determined by the path from $t=1$ to $t=0$. This pulls the branch point near $x = -1$ into the boundary at $x=0$, popping off a plane containing the corresponding branched point. This \emph{is} the Lefschetz singularity of $f_s$ and it contributes two Lefschetz singularities to the fibration in the double branched cover.

The entire description is summarized in Figure~\ref{fig:psi1}. However, to match with the skeleton used throughout, we should conjugate the factorization we achieved using the above perturbation by $\psi_1$ itself. We will see in Lemma~\ref{lem:hurwitz} that the two factorizations are Hurwitz equivalent and we write the monodromy factorization as $$\psi_1 = \tau_4 \tau_3 \tau_2 \tau_1 \tau_{1'}$$ where this labeling agrees with the labeling of the Dehn twist curves in Figure~\ref{fig:branched cover}.

\begin{figure}
    \centering
    \includegraphics[width=6in]{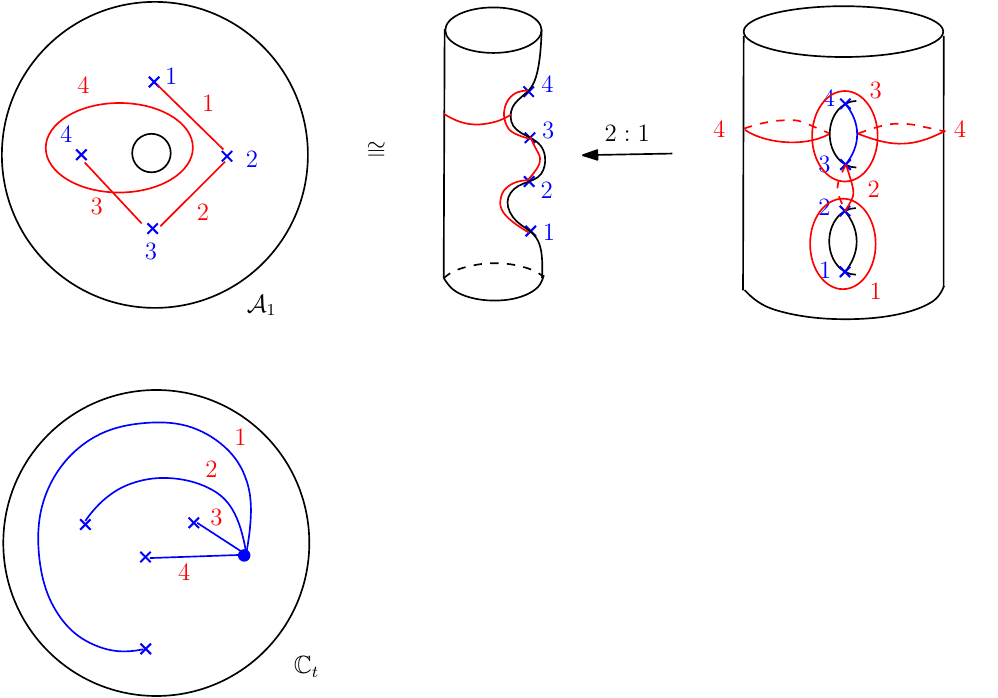}
    \caption{The paths in the $t$-plane, $\CC_t$, and corresponding branching and nodal singularity of $\Delta_s$ for $t = x(x^3-s-z)$ is shown on the left. In the middle, we identify this with the usual picture of the annulus as the quotient of the genus 2 curve. On the right, the factorization corresponding to the Lefschetz fibration on the double branched cover determined by $f_s$ on $\CC^2$. The labels $1, 2, 3, 4$ on the arcs and circles indicate the order in which the braid halftwists and Dehn twists occur in the factorization determined by the ordered choice of paths in $\CC_t$ shown in the bottom left. This skeleton on $\A_1$ doesn't agree with the description used in Figure~\ref{fig:branched cover}, but it turns out the corresponding word in the skeleton of Figure~\ref{fig:branched cover} agrees with this up to Hurwitz equivalence (see Lemma~\ref{lem:hurwitz}).}
    \label{fig:psi1}
\end{figure}

\subsection{\texorpdfstring{{\boldmath $t^2 = x(x^3-y^2)$}}{psi2}}
\label{sec:psi2}

The hypersurface $V$ in $\CC^3$ described by the polynomial $t^2 = x(x^3-y^2)$ is singular at the origin and we calculate the topology of its resolution $X$ in Section~\ref{sec:resolutions}, Figure \ref{VIalt}. The singular hypersurface again comes equipped with a genus 2 fibration (the fiber over the origin is singular) by intersecting with planes $t=\mathit{const}$ and this fibration extends to the resolution. We can deform the hypersurface $V$ into $\tilde{V}$ defined by the polynomial $t(t-s) = x(x^3-y^2)$, splitting the singularity into two singularities of type $\psi_1$, one at the origin and one at $t=s$. This is a smooth hypersurface of $\CC^3$ and using further deformations of the type used in the previous case of $\psi_1$, the corresponding fibration deforms into a Lefschetz fibration with word $$\psi_1^2 = \left(\tau_4 \tau_3 \tau_2 \tau_1 \tau_{1'}\right)^2.$$

However, we would like to know the stronger theorem, that $X$ and $\tilde{V}$ are deformation equivalent. To that end, we invoke Laufer's Theorem \ref{thm:laufer} to construct a very weak simultaneous resolution which produces a flat deformation of the resolution of $V$ into the resolution of $\tilde{V}$, calculating $K_s \cdot K_s$ by its value on the corresponding resolution $X_s$ of $V(f_s)$. From Proposition~\ref{prop:SV}, this amounts to ensuring that all $X_s$ are negative definite and with the same $b_1$ and $b_2$.

From the resolution we calculate in Section \ref{sec:resolutions}, Figure \ref{VIalt}, one can check that $b_1(X) = 0$, $b_2^+(X) = 0$ and $b_2^-(X) = 5$. The deformed surfaces $V_s$ are smooth and hence equal to their resolutions, and we can calculate $b_1(X_s)$, $b_2^+(X_s)$ and $b_2^-(X_s)$ from the Dehn twist factorization given above using the handle decomposition of $V_s$ described by the Lefschetz fibration \cite{GS}. The 4-manifold $V_s$ is built from $\Sigma_{2,2}$ by attaching 2-handles along each of the Dehn twist curves (with framing $-1$ relative to the fiber). First, we can see that the twists $\tau_2$, $\tau_3$, $\tau_4$, $\tau_{5}$ and $\tau_{5'}$ generate the first homology of the fiber $\Sigma_{2,2}$, so that $b_1(V_s)=0$. Second, $b_2(V_s) = 5$, as each of the remaining Dehn twists adds to the second homology. Third, the boundary 3 manifold is a rational homology sphere Seifert fibered space, so $b_2^0 = 0$, and thus we just need to show that the intersection form of $V_s$ is negative definite. One way to do this is to embed it into a different negative definite 4-manifold. If we attach a ``cap" to the fibration along $\bdry_2$ (topologically, this adds a 2-handle), we get the Lefschetz fibration for the singularity $\phi_2^2$ in \cite{SV} corresponding to the polynomial $y^2 = x(x^4 + t^2)$. There, we showed that the Lefschetz fibration corresponding to the monodromy $(\tau_1 \tau_1 \tau_2 \tau_3 \tau_4)^2$ is negative definite (with $b_2 = b_2^- = 6$). This is the monodromy (up to relabeling) of the capped factorization of $\psi_1^2$ and thus the intersection form for $V_s$ embeds as a sublattice of a negative definite lattice and so must also be negative definite. Thus $K_s^2$ is constant and equal to $K(X)^2$ and by Laufer has a simultaneous resolution, so there is some deformation value $s$ for which the Lefschetz fibration with word $\psi_1^2 = \left(\tau_4 \tau_3 \tau_2 \tau_1 \tau_{1'}\right)^2$ is a symplectic deformation of the fibration on the resolution $X$.

\subsection*{Smaller genus example: \texorpdfstring{{\boldmath $t^3 = x(x^2-y^2)$}}{psi0}} \label{sec:psi0}
The goal for this section is to study a smaller genus singularity that shows up in the splitting of the other singularities in the family $t^k = x(x^3-y^2)$. This fibration is related to the genus 1 \emph{star relation} \cite{Gervais}.

\noindent {\bf Mapping class group factorizations.}

\noindent Consider the surface $\Sigma_{1,1}$ of genus 1 and having 1 boundary component. This is the double branched cover of the disk with three marked points. Every mapping class is hyperelliptic and so the mapping class groups are the same $\MCG(\Sigma_{1,1}) \cong B_3$, where $B_3$ is the three stranded braid group. The abelianizations of these groups is $\ZZ$, with the isomorphism from $B_3$ being the algebraic word length of the braid (using the standard generators). In $\MCG(\Sigma_{1,1})$ there are only two types of right handed Dehn twists up to isomorphism: along homologically essential curves or boundary parallel. Under the quotient map from $\MCG(\Sigma_{1,1})$ to $\Ab(\MCG(\Sigma_{1,1}))$, every homologically essential twist maps to 1 and the boundary twist maps to 12 (as implied by the chain relation).

\begin{corollary} The only factorization in $\MCG(\Sigma_{1,1})$ of the boundary twist into a single twist must be by the boundary twist. Moreover, in any positive factorization of length strictly less than twelve, all twists save one must be trivial (i.e., along nullhomotopic curves in $\Sigma_{1,1}$). 
\end{corollary}

This holds more generally:

\begin{lemma} The only factorization of the single boundary multitwist $\tau_{\partial_1} \cdots \tau_{\partial_n}$ in $\MCG(\Sigma_{1,n})$ into $n$ positive twists must be the $n$ different boundary twists. 
\end{lemma}

\begin{proof} This follows directly from an inductive or iterative argument capping boundary components, but the statement is trivial if $n>9$. To see the latter fact, first note that any factorization of the $\tau_{\partial_1} \cdots \tau_{\partial_n}$ in $\MCG(\Sigma_{1,n})$ gives a Lefschetz pencil and the Lefschetz fibration obtained by blowing up the pencil locus is either $E(1)$ or $T^2\times S^2$. Such a fibration has $n$ disjoint sections, each a sphere of square $-1$. If any of the Dehn twists in the factorization are along non-separating curves, then the total space must be $E(1)\cong \CP\# 9\bCP$, which has at most 9 sections, and so $n\leq 9$. Additionally, when there is a Dehn twist along a non-separating curve, there must be at least 12 non-separating twists in the positive factorization, and both of these situations cannot happen at the same time. Thus the positive factorization uses only separating Dehn twists. 

We prove the following stronger statement with this additional hypothesis. 

\begin{claim} \label{lem:fact1n} The only factorization of the boundary multitwist in $\MCG(\Sigma_{1,n})$ into positive twists along separating curves must be the $n$ different boundary twists (plus trivial twists). \end{claim}

The corollary covers the base case of $n=1$ and we'll prove the other cases by induction. Assume the statement holds for $\MCG(\Sigma_{1,n})$ and consider any factorization of the boundary multitwist in $\Sigma_{1,n+1}$ into $\ell$ positive, separating twists. Capping one boundary component of $\Sigma_{1,n+1}$ yields a factorization of the boundary multitwist into separating, positive Dehn twists. By induction this must be the $n$ boundary twists plus some trivial twists along nullhomotopic curves. Every \emph{embedded} curve in $\Sigma_{1,n+1}$ which is nullhomotopic in $\Sigma_{1,n}$ is either nullhomotopic in $\Sigma_{1,n+1}$ or parallel to the boundary component which was capped. The fractional Dehn twist coefficient implies that we must have at most one boundary parallel curve in the factorization, so at most one of the trivial twists lifts to something nontrivial and each of the nontrivial twists either remains boundary parallel or encloses both its original boundary and the capped boundary. Repeating for each of the boundary components of $\Sigma_{1,n+1}$ we see that the latter case can never happen.
The factorization can use only twists parallel to the boundary components and so must be exactly one twist per boundary component.
\end{proof}

\begin{proposition}[Genus 1 Case: $t^3 = x(x^2-y^2)$] The monodromy factorization of the singularity $t^3 = x(x^2-y^2)$ is the boundary multitwist on a genus 1 surface with three boundary components.
\end{proposition}

\begin{proof} First, we show that the monodromy of the link of the singular fiber at $t=0$ is the boundary multitwist. We do this by first following the procedure laid out in ``base case" section \ref{sec:basecase} applied to the singularity $t = x(x^2-y^2)$. The monodromy for $t^3 = x(x^2-y^2)$ will be its third power and we will see that the monodromy of the link of $t^3 = x(x^2-y^2)$ is the boundary multitwist. We will then calculate the resolution of $t^3 = x(x^2-y^2)$, the topology of which will tell us that the factorization must be along three separating curves in $\Sigma_{1,3}$. Invoking Lemma~\ref{lem:fact1n} proves the proposition.

So to begin, we consider a smooth fiber of the simpler fibration $t = x(x^2-y^2)$. This is the curve $x(x^2-y^2)=t \subset \CC^2$. The link of this singularity is the (3,3) torus knot, and so the fiber is a genus 1 surface with three boundary components, $\Sigma_{1,3}$. Each fiber, smooth or not, is invariant under the involution $(x,y) \rightarrow (x,-y)$ and the quotient curve is $x(x^2 - z) = t \subset \CC^2 = \left<x,z\right>$, where the cover is given by $z = y^2$. The fixed point set of the involution is defined by the equation $y=0$ and its image under the quotient is $z=0$. The quotient fibration on $\CC^2$ is the Milnor fibration for the polynomial $x(x^2 - z)$. For $t\neq 0$, the fiber of $x(x^2 - z)$ is again an annulus which can be parametrized by $(x,x^2 - 1/x)$ for $x \in \CC^\times$. The branched cover has three branch points corresponding to the intersection with the curve $z=0$, namely the solutions to $x^3 = t$. When $t=0$, the fiber is a nodal curve consisting of two planes intersecting transversely (the line $x=0$ and the plane parametrized as the graph $(x,x^2)$). The point of intersection is also the only intersection point with the branch locus. The fibration by $x(x^2 - z)$ is then a Lefschetz fibration with a single Lefschetz critical point living above $t=0$. Using the deformation $(x(x^2 +s - z))$, we deform the singular fiber at $t=0$ so that is is transverse to the branch curve: one component has one of the branch points and the other has two. In the double cover then, one component is a disk and the other an annulus and they meet in two nodes. In particular, the cover from the $\Sigma_{1,3}$ is the one given in Figure~\ref{fig:psi0} where one boundary component covers 2:1 and the other two components cover 1:1. The vanishing cycle for the Lefschetz singular fiber at $t=0$ is shown, along with its cover.

We continue as before, determining the monodromy corresponding to the Lefschetz fibration on $\CC^2$ given by $t = x(x^2 +s -z)$. Choosing a small, negative (real) value for $s$ there are two branch points along the imaginary axis and a Lefschetz singular fiber over 0. In particular, continuing with the parametrization of the curve $t = x(x^2 +s -z)$ as $(x, x^2 + s - t/x)$ in the singular fiber over $t=0$, we can see which marked point bubbles off by itself, indicating which boundary component has a connected cover. The braid monodromy $b$ is shown in Figure~\ref{fig:psi0}, both factored as braid halftwists and Dehn twists and also via a partial cut system showing an arc in the annulus and its image after the monodromy. This more easily allows us to see $b^3$, the braid monodromy around the link of the singularity for $t^3 = x(x^2-z)$, and write it as three Dehn twists as shown. In the double cover, we see that $b^3$ lifts to a product of boundary parallel Dehn twists. Thus the monodromy of the singularity $t^3 = x(x^2-y^2)$ can be factored as a boundary multitwist. 

We calculate the resolution in Section~\ref{sec:g=1} and see that $b_1 = 2$, $b_2 = 1 = b_2^-$, so any factorization corresponding to the resolution must have three Dehn twists. By Lemma~\ref{lem:fact1n}, the only factorization of the boundary multitwist into three twists is the boundary single multitwist $\tau_\bdry$ and so this is the monodromy factorization corresponding to the Lefschetz fibration on the resolution. 

(One could go further, calculating the Lefschetz fibration for $t = x(x^2-y^2)$ as a word $\psi_0$ with that $\psi_0^3 = \tau_\bdry$. And indeed this substitution corresponds to the \emph{star relation} of Gervais \cite{Gervais}.)
\end{proof}

\begin{figure}
    \centering
    \includegraphics[width = 3in]{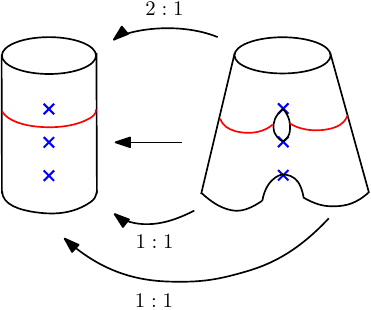}\\
    \vspace{0.2in}
    \includegraphics[width = 4.5in]{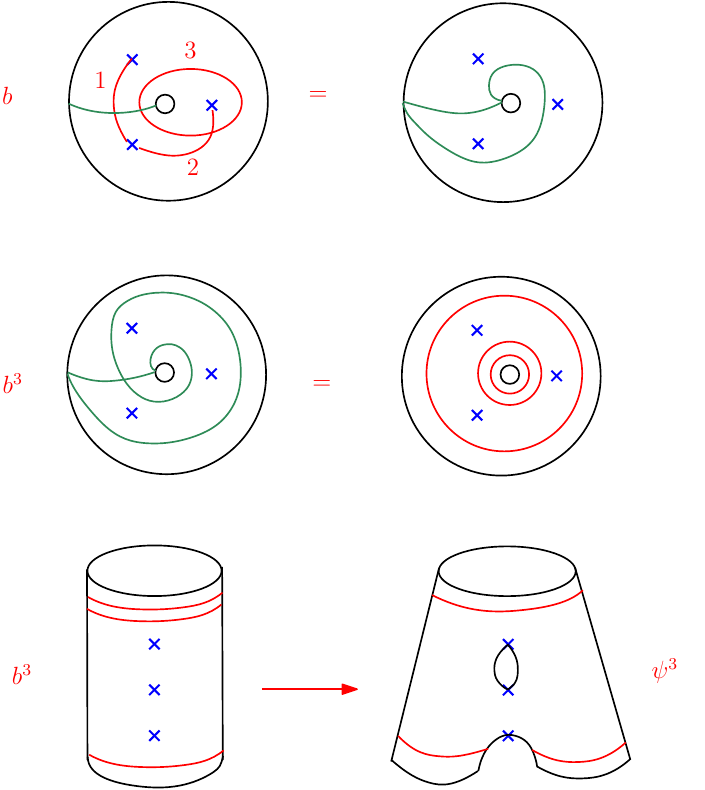}
    \caption{The cover and monodromy corresponding to the singular fibration on $t = x(x^2-y^2)$. The labels $1, 2, 3$ on the arcs and circles indicate the order in which the braid halftwists and Dehn twists occur in the factorization determined by the ordered choice of paths in $\CC^t$ specified in the proof.}
    \label{fig:psi0}
\end{figure}

\subsection{\texorpdfstring{{\boldmath $t^3= x(x^3-y^2)$}}{psi3}}
\label{sec:psi3}
To determine the monodromy factorization of the resolution $X$ of the singularity $V$ defined by $t^3 = x(x^3-y^2)$ we deform to split off a singularity of type $t^3 = x(x^2 - y^2)$ (the example discussed above) via the family $V_s = V(t^3 - x(x^2(x+s) + y^2)$ and track everything in the quotient $t^3 = x(x^2(x+s) -z)$. As in the case of $\psi_1^2$, the hypersurface $t^3 = x(x^2(x+s) -z)$ is singular over $t=0$ and is fibered by annuli $\A_t$ which we think of as curves in the $(x,z)$ plane. For $t\neq 0$, $\A_t$ can be parametrized by $(x,x^2(x+s) -t^3/x)$ for $x \in \CC^\times$. $V(s)$ is the branched double cover of the hypersurface $t^3 = x(x^2(x+s) -z)$ branched along the curve $\Delta_s$ defined by $z=0$. When $s=0$, $\Delta$ intersects all fibers $A_t$ with $t\neq 0$ in four points. Looking in $\CC^\times$, those four points are the solutions to $x^2(x+s) - t^3/x = 0$. The branch locus $\Delta$ is braided with respect to the annulus fibration and if we choose the deformation value $s$ to be a small positive real number, $\Delta$ has three points of braiding which occur along $t$-values with angular coordinate $\pi/3$, $\pi$ and $5\pi/3$. Each braid point  will contribute a Lefschetz singularity to the branched double cover which we can determine by finding the arc along which the braiding occurs.

The singular fiber over $t=0$ is complicated. It has two components, a curve given by $x=0$ with one branch point (at the origin), and a curve $x^2(x+s) -z$ with two branch points, one of multiplicity 2 at $x=0$ and one of multiplicity 1 at $x = -s$. Near the origin, $V_s$ has a singularity of type $t^3 = x(x^2 - y^2)$, and we have already seen the local model of the resolution and its deformation into a Lefschetz fibration in Section \ref{sec:psi0}. Call the corresponding resolution $X_s$. 

We pull each of these back to the reference fiber $t=1$ along the arcs shown in Figure~\ref{fig:psi13} which gives the factorization both of the quotient and the cover. To identify the particular subsurface corresponding to the $t^3 = x(x^2 - y^2)$ singularity, we watch the movie along path number 4 from $t=1$ to $t=0$ and we see that the points labeled 1, 2 and 3 converge to $x=0$ and exit the fiber along $\partial_1$. 

Putting the picture together, we see the monodromy factorization corresponding to the Lefschetz fibration on the resolution $X_s$. Due to its similarity to the base case, we denote the factorization by $\psi_{\tilde{1}}$ where
 \[\psi_{\tilde{1}} = \tau_2 \tau_3 \tau_4 \tau_5 \tau_{5'} \tau_{\partial_1}\]

As before, we would like this deformation to lift to a deformation on the original resolution $X$ and as before we will use Laufer's theorem (Theorem~\ref{thm:laufer}). The resolution determined in Section~\ref{sec:resolutions} has $b_1 = 0$ and $b_2 = b_2^- = 1$ (see Figure \ref{table2}, the first row). Using the Kirby diagram associated to the monodromy factorization $\psi_{\tilde{1}}$ and canceling handles (see \cite{GS}), we see that $X_s$ has the same values of $b_1$, $b_2$ and $b_2^{-}$ (and indeed that the generator of $H_2$ has self-intersection $-3$) and so the deformation on $V_s$ lifts to a deformation on the resolution. In Section~\ref{sec:negativedef}, we will show that the total spaces of the Lefschetz fibrations associated to all the words used in Theorem~\ref{thm:main2} are negative (semi-)definite. That argument applies here as well.

\begin{figure}
    \centering
    \includegraphics[width=6in]{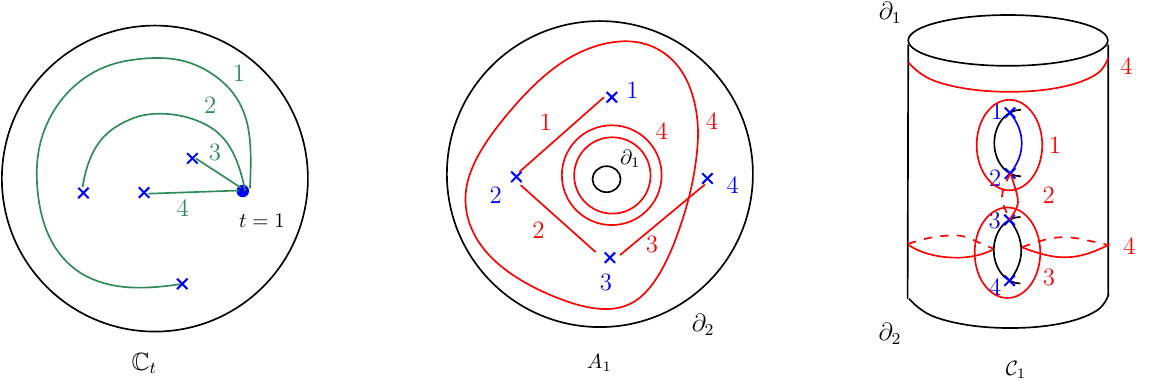}
    \caption{Paths in the $t$-plane used to find the factorization $\psi_{\tilde{1}}$ of $\psi_1^3$, along with the braid halftwists and Dehn twists used to identify the factorization in the cover. The singularity over $t=0$ corresponds to a full twist about the enclosed points along with two twists along the inner boundary $\partial_1$ and lifts to three Dehn twists that separate off a genus one surface with three boundary components (one of which is $\partial_1$). The labels $1, 2, 3, 4$ on the arcs and circles indicate the order in which the braid halftwists and Dehn twists occur in the factorization determined by the ordered choice of paths in $\CC_t$ shown on the left.}
    \label{fig:psi13}
\end{figure}

\begin{lemma} \label{lem:hurwitz} The monodromy factorizations $\psi_1$ and $\psi_1^k(\psi_1)= \psi_1^{-k} \psi_1 \psi_1^k$ are Hurwitz equivalent (in the strong sense, no cyclic permutation is needed). Similarly $\psi_{\tilde{1}}$ and $\psi_1^k(\psi_{\tilde{1}})= \psi_1^{-k} \psi_{\tilde{1}} \psi_1^k$.
\end{lemma}

\begin{proof}
In the base case of $t=x(x^3 - y^2)$ of Section~\ref{sec:basecase} we showed that the factorization for $\psi_1$ is the double branched cover of the braid factorization shown in Figure~\ref{fig:psi1} and also that the monodromy diffeomorphism of the braid consists of a clockwise quarter turn of the four points (finished with a remaining 3/4 twist as you approach $\partial_1\A$). We also mention that the skeleton used in the factorization $\tau_4 \tau_3 \tau_2 \tau_1 \tau_{1'}$ does not agree with the skeleton shown in Figure~\ref{fig:branched cover}. In fact, if we denote by $\phi$ the monodromy found in Section~\ref{sec:basecase}, then $\phi = \psi_1 \psi_1 \psi_1^{-1}$. With that in mind, we work with the description  $\psi_1 = \tau_4 \tau_3 \tau_2 \tau_1 \tau_{1'}$ using the standard skeleton of Figure~\ref{fig:branched cover}, noting a posteriori that this implies that the monodromy of the singularity  $t=x(x^3 - y^2)$ can be written as $\psi_1 = \tau_4 \tau_3 \tau_2 \tau_1 \tau_{1'}$ (again using the standard skeleton).

To see the relationship between $\psi_1$ and $\psi_1(\psi_1) =  \psi_1^{-1} \psi_1 \psi_1$, we begin with the factorization of the braid monodromy of $\psi_1$ as $\b_1 \b_2 \b_3 \t$ on the annulus, using the standard skeleton from Figure~\ref{fig:branched cover}, where $\b_i$ is the braid halftwist along the (red) arc labeled $i$. First we slide the full twist $\t$ past the third braid halftwist $\b_3$, bringing it to the marked point on the left. Next we slide $\b_3$ back over the new Dehn twist $\t'$, and then past $\b_2$ and $\b_1$. This results in a new factorization $\b_{3'} \b_1 \b_2 \t'$ which agrees with the factorization $\psi_1(\psi_1)$ on the double branched cover. The sequence of factorizations is show in the first part of Figure~\ref{fig:hurwitz-conjugations}. The same sequence of moves when applied to $\psi_{\tilde{1}}$ yields $\psi_1 (\psi_{\tilde{1}})$ and the corresponding diagrams are given in the same figure.
\end{proof}

\begin{figure}
    \centering
    \includegraphics[width = 5.5in]{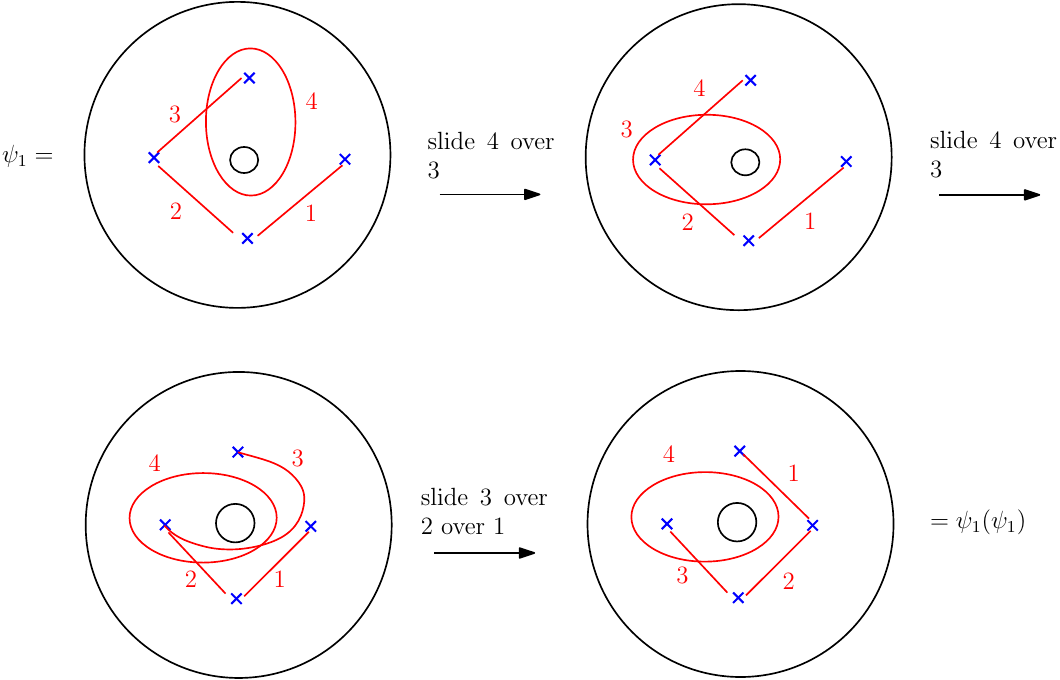}\\
    \vspace{0.4in}
    \includegraphics[width = 5.5in]{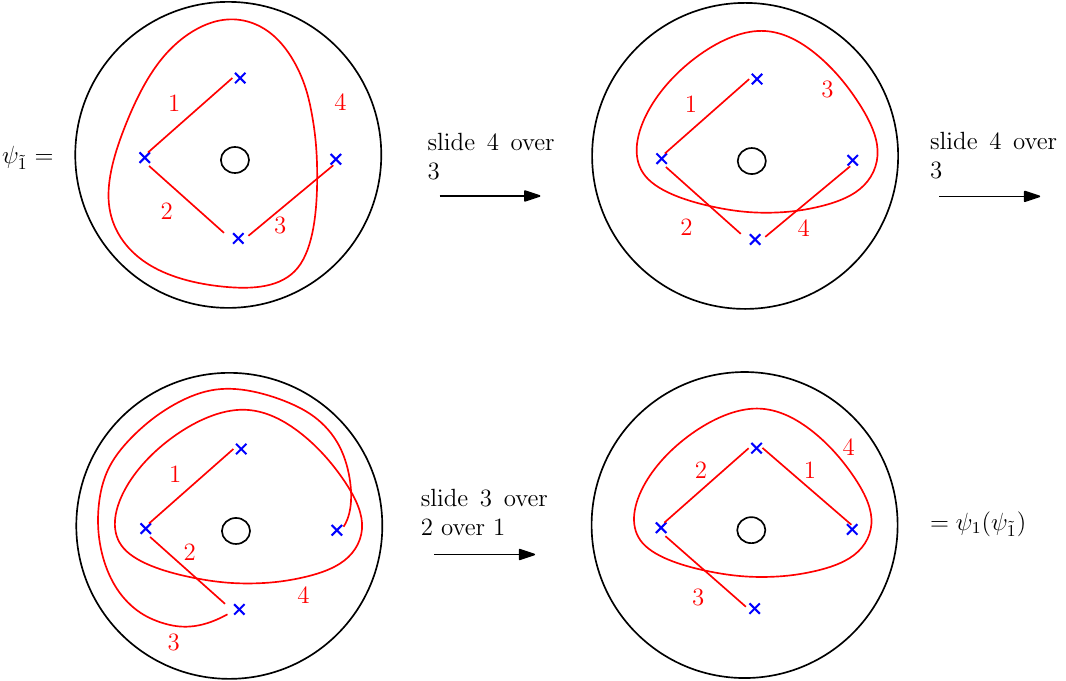}
    \caption{On the top, the sequence of Hurwitz moves that relates $\psi_1$ to $\psi_1(\psi_1)$. On the bottom the same for $\psi_{\tilde{1}}$ and $\psi_1(\psi_{\tilde{1}})$. The labels $1, 2, 3, 4$ indicate the order in which the braid halftwists and Dehn twists occur in the corresponding factorization. }
    \label{fig:hurwitz-conjugations}
\end{figure}

\subsection{\texorpdfstring{{\boldmath $t^4 = x(x^3-y^2)$}}{psi4}}
\label{sec:psi4}

Given the previous work, this and all following cases are fairly straight-forward. Using the deformation $t^3(t+s) = x(x^3-y^2)$, splits this into a type $\psi_{\tilde{1}}$ singularity at $t=0$ and a type $\psi_1$ singularity at $t=-s$. Depending on the path chosen from the reference fiber at $t=1$ to the singular fibers at $t=0$ and $t=-s$, we would get monodromy factorizations of a product of conjugates of $\psi_{\tilde{1}}$ and $\psi_1$. However, different paths yield conjugates by powers of $\psi_1$ and by Lemma~\ref{lem:hurwitz}, those are all Hurwitz equivalent. Thus we see the factorization

\[\psi_1^4 = \psi_{\tilde{1}} \psi_{{1}} = \tau_2 \tau_3 \tau_4 \tau_5 \tau_{5'} \tau_4 \tau_3 \tau_2\tau_1\tau_{1'} \tau_{\partial_1}\]

This factorization has length 11, yielding a 4-manifold with $b_1 =0$ and $b_2 = 6$. We will show later that the 4-manifold is negative definite (see Section~\ref{sec:negativedef}), and since these values agree with the resolution calculated in Section~\ref{sec:resolutions} (Figure \ref{table2}, graph of $\psi^4$), this along with Proposition~\ref{prop:Laufer} is sufficient to conclude that the deformations given yield a flat deformation of the resolution $X$ into the Lefschetz fibration associated to the monodromy factorization above.

\subsection{\texorpdfstring{{\boldmath $t^5= x(x^3-y^2)$}}{psi5}}
\label{sec:psi5}
We follow a similar deformation as was used in the previous case, $t^3(t+s)(t+2s) = x(x^3-y^2)$ to get the factorization

\[\psi_1^5 = \psi_{\tilde{1}} \psi_{{1}}^2.\]

The arguments above can be applied here as well and we conclude that the Lefschetz fibration corresponding to the factorization on the right is indeed a deformation of the resolution.

\subsection{\texorpdfstring{{\boldmath $t^6= x(x^3-y^2)$}}{psi6}}
\label{sec:psi6}
In this case the deformation $t^3(t+s)^3 = x(x^3-y^2)$ splits this into two type $\psi_{\tilde{1}}$ singularities. The corresponding values of $b_1$ and $b_2 = b_2^-$ agree with the resolution and so 

\[\psi_1^6 = \psi_{\tilde{1}}^2.\]

\subsection{\texorpdfstring{{\boldmath $t^7= x(x^3-y^2)$}}{psi7}}
\label{sec:psi7}
Here we use the deformation $t^3(t+s)^3(t+2s) = x(x^3-y^2)$ splits this into two type $\psi_{\tilde{1}}$ singularities and a type $\psi_1$ singularity. The corresponding values of $b_1$ and $b_2 = b_2^-$ agree with the resolution and so 

\[\psi_1^7 = \psi_{\tilde{1}}^2\psi_{{1}}.\]

\subsection{\texorpdfstring{{\boldmath $t^8= x(x^3-y^2)$}}{psi8}}
\label{sec:psi8}

As a monodromy, the quotient braid associated to $\psi_1^4$ is a full-circle push map of the four marked points plus some twists around the boundary components. Tracing out the image of an arc (as we did in Section~\ref{sec:psi0}), shows that $\psi_1^4$ is a right-handed push map (as measured from $\partial_2$) along with four right-handed Dehn twists about $\partial_1$. Together this gives a factorization of the braid for $\psi_1^4$ as $\tau_{\partial_1}^3 \tau_{\partial_2}$, and hence $\psi_1^8 = \tau_{\partial_1}^6 \tau_{\partial_2}^2$. Since the cover $\Sigma_{2,2} \mapsto A$ wraps each boundary component of $\Sigma_{2,2}$ twice around the corresponding boundary component of $A$, neither $\tau_{\partial_1}$ nor $\tau_{\partial_2}$ (on $A$) lift to diffeomorphisms of $\Sigma_{2,2}$, but $\tau_{\partial_1}^2$ lifts to $\tau_{\partial_1}$ and $\tau_{\partial_2}^2$ lifts to $\tau_{\partial_2}$. Thus as a mapping class element $\psi_1^8 = \tau_{\partial_1}^3\tau_{\partial_2}$ on $\Sigma_{2,2}$. In Section~\ref{sec:resolutions} we showed that the resolution of this singularity has $b_1 = 4$ and $b_2 = b_2^{-} = 3$ (cf. the last row of Figure \ref{table2}), and thus corresponds to a monondromy factorization consisting of four separating curves. The following lemma then shows that the monodromy factorization of this singularity must be

\[\psi_1^8 = \tau_{\partial_1}^3\tau_{\partial_2}.\]

\begin{lemma} The only factorization of the mapping class  $\tau_{\partial_1}^3\tau_{\partial_2}$ on $\Sigma_{2,2}$ into four Dehn twists along separating curves is $\tau_{\partial_1}^3\tau_{\partial_2}$.
\end{lemma}
\begin{proof}
Capping both boundary components of $\Sigma_{2,2}$ yields a factorization of the identity into four separating Dehn twists. We will show that all Dehn twists in such a factorization have to be trivial. The abelianization of the mapping class group of the closed genus 2 surface $\Sigma_2$ is $\ZZ/10 \ZZ$ (\cite{Korkmaz}), and there is an identification of the two groups in which a Dehn twist about a homologically essential curve is identified with 1. Using the genus 1 chain relation, we can see the Dehn twist about a curve that separates $\Sigma_2$ into two genus 1 surfaces is identified with 2. A trivial Dehn twist is then 0. The only way to write 0 in $\ZZ/10 \ZZ$ as the sum of 4 numbers, each of which is 0 or 2, is as 4 copies of 0. Hence any factorization of the identity into 4 separating Dehn twists must be along trivial (null homotopic) curves. Thus all the twists in the original factorization must separate off either one or both boundary components from the rest of the genus 2 surface. Capping $\partial_1$, each curve is then either parallel to $\partial_2$ or trivial. This yields a factorization of the boundary twist $\tau_{\partial_2}$ and so must consist of exactly one twist parallel to $\partial_2$. Thus only one curve can enclose $\partial_2$ while the others can enclose only $\partial_1$. Capping $\partial_2$ shows that this curve must be parallel to $\partial_2$. Thus the only factorization of the monodromy $\tau_{\partial_1}^3\tau_{\partial_2}$ into four right-handed twists is that exact factorization. 
\end{proof}

\subsection{A version of the hyperelliptic monodromy factorization}
\label{sec:hyperelliptic}

There is a common monodromy factorization of the hyperelliptic involution on a genus two surface $\Sigma_2$. This factorization is
\[ I = \tau_1 \tau_2 \tau_3 \tau_4 \tau_5 \tau_5 \tau_4 \tau_3 \tau_2 \tau_1\]
and $I^2 = id$.

There are lifts of this relation to $\Sigma_{2,1}$ and $\Sigma_{2,2}$ where the quotient is a disk with either 5 or 6 branch points (respectively), and its centralizer determines the usual notion of the hyperelliptic mapping class group in each case, and each lift satisfies $\tilde{I}^2 = M_\partial$ (where $M_\partial$ is the (multi)twist consisting of a single positive twist about each boundary component). 

The monodromy factorization above for $\psi_1^4 = (\tau_2 \tau_3 \tau_4 \tau_5 \tau_{5'} \tau_4 \tau_3 \tau_2 \tau_1 \tau_{1'} \tau_{\partial_1})$ tells that there is a second lift to the mapping class group of $\Sigma_{2,2}$ which generates an involution whose quotient is an \emph{annulus}. (Indeed, this is the same quotient giving the branched covers above.)

Let 
\[ \tilde{I} = \tau_2 \tau_3 \tau_4 \tau_5 \tau_{5'} \tau_4 \tau_3 \tau_2 \tau_1\tau_{1'}\]

so that $\tilde{I} = \psi_{\tilde{1}} \psi_1 \tau_{\partial_1}^{-1}$.

We saw in Section~\ref{sec:psi8} that in the mapping class group $(\psi_{\tilde{1}} \psi_1)^2 = \tau_{\partial_1}^3 \tau_{\partial_2}$ which implies that $\tilde{I}^2 = \tau_{\partial_1} \tau_{\partial_2}$. Thus $\tilde{I}$ is a different lift of the hyperelliptic involution, the one that characterizes those mapping class group elements that arise as lifts of elements of the marked group of the annulus. In summary we proved \cref{newlift}, which we restate below.

\begin{repprop}{newlift} The following relation holds in the mapping class group of the surface $\Sigma_{2,2}$
\[(\tau_2 \tau_3 \tau_4 \tau_5 \tau_{5'} \tau_4 \tau_3 \tau_2 \tau_1\tau_{1'})^2 = \tau_{\partial_1} \tau_{\partial_2}\]
and moreover, the mapping class element represented by $\tilde{I} = \tau_2 \tau_3 \tau_4 \tau_5 \tau_{5'} \tau_4 \tau_3 \tau_2 \tau_1\tau_{1'}$ is isotopic to the involution on $\Sigma_{2,2}$ with four fixed points whose quotient is the annulus. 
\end{repprop}

After we found this relation, N. Monden told us that this was known to them (\cite{AM}, Proposition 41), but our method is different.

\subsection{A discussion of negative definiteness} \label{sec:negativedef}

As resolutions of hypersurface singularities, each of the resolutions constructed in Section \ref{sec:resolutions} are negative definite. To apply Laufer easily in the latter case above, we would like to show that the 4-manifolds that correspond to the monodromy factorizations used in the deformations of $t^k = x(x^3-y^2)$ for $k=3,\dots,7$ are also negative definite, something which we assumed earlier and which we will prove now. This along with the interpretation of Laufer's theorem in Proposition~\ref{prop:Laufer} allows us to quickly conclude that the deformations we use in the following cases are flat deformations of the resolved complex surface. We note that the monodromy factorizations that we will use for these cases are $\psi_{\tilde{1}} \psi_1$, $\psi_{\tilde{1}} \psi_1^2$, $\psi_{\tilde{1}}^2$ and $\psi_{\tilde{1}}^2 \psi_1$. By Lemma~\ref{lem:hurwitz}, each of these is a subword of the word $\psi_{\tilde{1}}^2 \psi_1^2 = \tilde{I}^2 \tau_{\partial_1}^2$. Even though the cases $k=1,2$ were handled independently, as $\psi_1$ and $\psi_1^2$ are also subwords of $\tilde{I}^2 \tau_{\partial_1}^2$, the argument works equally well there.
We will show that the 4-manifold $X_0$ associated to the word $\psi_{\tilde{1}}^2 \psi_1^2$ has a negative semi-definite intersection form, which then implies the following proposition.

\begin{proposition} \label{prop:negdef} The 4-manifolds associated to the monodromy factorizations of $t^k = x(x^3 - y^2)$ for $k=1,\dots,7$ are negative definite.
\end{proposition}

\begin{proof} 
As discussed in the preamble to the proposition, let $X_0$ be the 4-manifold associated to the positive word $\psi_{\tilde{1}}^2 \psi_1^2 = \tilde{I}^2 \tau_{\partial_1}^2$.
Observe that we can embed $X_0$ into a closed Lefschetz fibration by capping both boundary components of $\Sigma_{2,2}$ and gluing in $\Sigma_2 \times D^2$. In doing this, each boundary component of $\Sigma_{2,2}$ yields a section of the corresponding fibration. Moreover, this fibration is the one associated to the positive word $\psi_{\tilde{1}}^2 \psi_1^2 = I^2 \tau_{\partial_2}^2$ thought of as acting on $\Sigma_2$ and it is also the fibration for $I^2$ blown up twice. The fibration $I^2$ produces the closed 4-manifold $\CP\#13 \bCP$ and from Section~\ref{sec:hyperelliptic} we see that this has two sections of square -1. To get $X_0$, we blow up this fibration twice along the section corresponding to $\partial_2$ and then remove the two sections and a regular fiber. Letting $N_0$ be a neighborhood of those three surfaces, we can summarize this discussion as $\CP\#15 \bCP = X_0 \cup_{Y_\partial} N_0$. The boundary 3-manifold $Y_\partial$ is a Seifert fibered space over $\Sigma_2$ with a single singular fiber and so has $b_1 = 4$. We check that $b_1(N_0) = 4$, $b_2(N_0) = 1$, $b_2^+(N_0) = 1$, and $b_2^-(N_0) = 2$, yielding $\sigma(N_0) = -1$. Since $\sigma(\CP\#15 \bCP ) = -14$, by Novikov additivity, we have $\sigma(X_0) = -13.$ The word associated to $X_0$ has length 22, so $X_0$ is built from $\Sigma_{2,2}\times D^2$ by attaching twenty-two four-dimensional 2-handles. Five of those handles kill the 1-handles coming from $\Sigma_{2,2}$ and the other seventeen contribute to $b_2(X_0) = 17.$ Since $b_1(X_0) = 0$, $b_2^0(X_0) = 4$ and since $\sigma(X_0) = -13$, the remaining $b_2$ must all be negative definite: $b_2^+(X_0)=0$ and $b_2^{-}(X_0)=13$. Thus each of the words $\psi_{\tilde{1}} \psi_1$, $\psi_{\tilde{1}} \psi_1^2$, $\psi_{\tilde{1}}^2$ and $\psi_{\tilde{1}}^2 \psi_1$ yields a negative semi-definite 4-manifold. Finally, since each of the 3-manifolds arising as the boundary of one of $t^k = x(x^3 - y^2)$ for $k=1,\dots,7$ are rational homology spheres, $b_2^0=0$ so these 4-manifolds must be negative definite.
\end{proof}

\bibliography{References}
\bibliographystyle{amsalpha}

\vspace{0.3in}
\end{document}